\documentclass[12pt, reqno]{amsart}
\usepackage{amssymb}
\usepackage{amsfonts}
\usepackage{amssymb,amsmath,amscd}
\usepackage[colorlinks=true]{hyperref}
\usepackage{mathrsfs}
\newtheorem{Theorem}{\indent Theorem}[section]

\newtheorem{Lemma}[Theorem]{\indent Lemma}
\newtheorem{Corollary}[Theorem]{\indent Corollary}

\theoremstyle{remark}
\newtheorem{Remark}{Remark}

\begin{document}
\centerline{
\bf On squares associated to the Piatetski-Shapiro sequences}
\bigskip
\centerline{\small  Wei Zhang}
\bigskip
\centerline{School of Mathematics and Statistics, Henan University, zhangweimath@126.com}

\centerline{https://orcid.org/0000-0002-5505-2671}

\bigskip


\textbf{Abstract}
We can give a slight improvement for  a result related to the distribution of squares in Piatetski-Shapiro sequences of Liu-Shparlinski-Zhang \cite{lsz} by involving three different ideas.



\bigskip
\bigskip
\numberwithin{equation}{section}

\section{Introduction}
In this paper, we will   further consider the issue related to the squares in the   Piatetski-Shapiro sequences
$
\left( [n^{c}]\right)_{n\in \mathbb{N}},\ \textup{for}\ c>1\   \ \textup{and}\  c\notin \mathbb{N}.
$
A number $q$ is called
square-free integer  if and only if (see  \cite{MV})
$
m^{2}|q\Longrightarrow m=1.
$
It
is well known that in terms of sufficiently large $x,$ one has
$
\sum_{n\in\mathcal{Q}_{2}}n^{-s}=
\zeta(s)/\zeta(2s),\  \sigma>1,
$
and
$
\sum_{n\leq x,\ n\in\mathcal{Q}_{2}}1=
x/\zeta(2)+O(x^{1/2}),
$
where $\mathcal{Q}_{2}$ is
the set of positive square-free integers.
Fixed $c\in(1,3/2),$ following a paper
of Stux \cite{St}, by using some results of Deshouillers \cite{De}, in 1978, Rieger \cite{Re}
proved that
$
\sum_{n\leq x,\ [n^{c}]\in\mathcal{Q}_{2}}1= 6x/\pi^{2}
+O(x^{c/2+1/4+\varepsilon}).
$
This was improved and generalized by many experts ( for example,  see \cite{BBB,CZ,CZ1}).
In \cite{CZ1}, Cao and Zhai proved that
$\sum_{n\leq x,\ [n^{c}]
\in\mathcal{Q}_{2}}1=6x/\pi^{2}
 +O(x^{1-\varepsilon}),$
where $c\in(1,149/87)$ and
$\varepsilon$
is an arbitrarily small constant.

Recently, the distribution of squares in a Piatetski-Shapiro sequence was considered by  Liu-Shparlinski-Zhang \cite{lsz}. For a real number $1<c<2$ and positive integers $N$ and $s$, let $Q_{c}(s,N)$ denote the number of solutions to the equation
\[
[n^c]=sm^2,\quad 1\leq n\leq N, \quad m,n\in \mathbb{Z}.
\]
Then we have
$Q_{c}(s,N)=\sum_{n\leq N\atop [n^{c}]=s\square}1,
$
 where $\square$ denote a non-specified square. The results of
 Liu-Shparlinski-Zhang \cite{lsz} imply that for $\gamma=1/c,$  any $c\in(1,2)$ and any exponent pair $(\kappa,\lambda),$ as $N\rightarrow\infty,$ there exists an asymptotic formula
\begin{align*}
\begin{split}
Q_{c}(s,N)&=\gamma(2\gamma-1)^{-1}
s^{-1/2}N^{1-c/2}\\
& +O\left(s^{-\rho_{1}(\kappa,\lambda)}
N^{\mathcal{\theta}_{1}(c,\kappa,\lambda)+\varepsilon}
+s^{-\rho_{2}(\kappa,\lambda)}
N^{\mathcal{\theta}_{2}(c,\kappa,\lambda)+\varepsilon}
\right),
\end{split}
 \end{align*}
where
\[
\rho_{1}(\kappa,\lambda)=\frac{\lambda}{2(1+\kappa)}
,\ \ \ \ \rho_{2}(\kappa,\lambda)=\frac{\lambda-\kappa}{2}
\]
and
\[
\mathcal{\theta}_{1}(c,\kappa,\lambda)=\frac{c\lambda+2\kappa}{2(1+\kappa)},\ \ \ \ \ \mathcal{\theta}_{2}(c,\kappa,\lambda)=\frac{c(\lambda-\kappa)+2\kappa}{2}
.\]
A minor refined result was given in \cite{zw}.
Moreover, an averaged type sum was also considered by Liu-Shparlinski-Zhang \cite{lsz}.
 Let
$
\mathfrak{Q}_{c}(S,N)=\sum_{\substack{s\leq S\\ s\in\mathcal{Q}_{2}}} Q_{c}(s,N),
$
which can be seen as $Q_{c}(s,N)$  on average over positive square-free integers $s\leq S$. Obviously, we can observe that only the case $S\leq N^{c}$ is meaningful. So we always assume this. Nontrivial upper bounds for  $\mathfrak{Q}_{c}(S,N)$ correspond to a lower bound on the number square-free parts of the integers $[n^{c}],$ $1\leq n\leq N.$ Equally, this can also be formulated as a lower bound on the number of distinct quadratic fields in the sequence  of fields $\mathbb{Q}\left(\sqrt{[n^{c}]}\right).$
As to $\mathfrak{Q}_{c}(S,N)$, their arguments in \cite{lsz} imply that for any $c\in(1,2),$ $\gamma=1/c$ and  $N\rightarrow\infty$, one has
\begin{align}\label{qq}
\begin{split}
\mathfrak{Q}_{c}(S,N)&-
\frac{12\gamma}{\pi^{2}(2\gamma-1)}
S^{1/2}N^{1-c/2}\\
&\ll S^{1/5}N^{(1+2c)/5+\varepsilon}
+S^{1/8}N^{(2+3c)/8+\varepsilon}
+S^{5/8}N^{3c/8+\varepsilon}
+SN^{1-c+\varepsilon}.
\end{split}
\end{align}
 By introducing some different ideas and inserting similar ideas of Corollary 6.7 in \cite{Bor}, we can  show the follows. We can eliminate two error terms for the above asymptotic formula (\ref{qq}) for certain range of $c$.
\begin{Theorem}\label{LSZ0}
For any fixed $c\in(1,2),$ $\gamma=1/c,$   any exponent pair $(\kappa,\lambda)$ with $(1-\kappa)/2\leq\lambda-\kappa$   and  $N\rightarrow\infty$,  we have
\begin{align*}
\begin{split}
\mathfrak{Q}_{c}(S,N)&-
\frac{12\gamma}{\pi^{2}(2\gamma-1)}
S^{1/2}N^{1-c/2}\\
&\ll \min\left\{S^{1/5}N^{(1+2c)/5+\varepsilon},
S^{1/8}N^{(2+3c)/8+\varepsilon},
S^{\frac{\lambda}{1+\kappa}-\frac12}
N^{\frac{\kappa}{1+\kappa}+\frac c2+\varepsilon}+N^{\varepsilon}S^{1-\gamma}\right\}
\\&+SN^{1-c+\varepsilon}.
\end{split}
\end{align*}
\end{Theorem}
\begin{Remark}
In \cite{lsz}, the authors pointed out that  one cannot improve their result only by using some recent bounds of Robert and Sargos \cite{RS} (Theorem 1), for that there are exactly the   terms
that dominate in their applications.   The first  idea is from \cite{zw,CZ} such that   we can find a suitable exponent pair $(\kappa,\lambda)$ such that $\lambda/(1+\kappa)=1/2.$ Hence, the   term $S^{\frac{\lambda}{1+\kappa}-\frac12}
N^{\frac{\kappa}{1+\kappa}+\frac c2+\varepsilon}$ in the asymptotic formula of Theorem  \ref{LSZ0} is independent of $S.$ The second idea is that we can change quadruple sum into  triple sum by interesting the classical divisor function. This is possible. Because  there
two variables have the same indexes. The third idea is that
a coefficient involving one variable is 1. This allows us to   use Theorem 3 in \cite{RS}.
\end{Remark}
Corollary 2.4 (as a corollary of (\ref{qq})) in \cite{lsz} implies that for any $c\in(1,2),$ any $\varepsilon>0$ and $S\leq N^{\tau(c)-\varepsilon},$
one has
$
\mathfrak{Q}_{c}(S,N)=o(N)
$
as $N\rightarrow\infty,$
where
\begin{align} \label{key}
\tau(c)=\begin{cases}
(8-3c)/5,\ \ &\textup{for}\ 1<c\leq 12/7,\\
2(2-c),\ \ &\textup{for}\  c\geq 12/7.
\end{cases}
\end{align}
As cited in \cite{lsz}, only the cases $S\leq N^{c}$ is meaningful. Hence the best possible range of $S$ should be $S\leq N^{c}.$
With Theorem \ref{LSZ0} and choosing $(\kappa,\lambda)=BABAAB(0,1)=(4/18,11/18)$, we can obtain the follows.

\begin{Corollary}\label{LSZ2}
For  any $\varepsilon>0$,
one has
$
\mathfrak{Q}_{c}(S,N)=o(N)
$
as $N\rightarrow\infty,$
where
\begin{align}
\tau(c)=\begin{cases}
c-\varepsilon,\ \ &\textup{for}\ 1<c\leq 18/11,\\
3(2-c),\ \ &\textup{for}\ 18/11<c<2.
\end{cases}
\end{align}
\end{Corollary}
Obviously, for fixed $c\in(1,18/11)$, we enlarge the effective range of $S$ given by (\ref{key}). It is worth emphasizing that the range of $S$ given by Corollary
\ref{LSZ2} on $c\in(1,11/8)$ is almost best possible.

\begin{Remark}
A similar remark can also be seen in the author's article \cite{zw}. We quote here for complementary.
Assuming that $(a,1/2+a)$  is an exponent  pair, then
we can get $
\mathfrak{Q}_{c}(S,N)=O(N^{1-\varepsilon})
$
with $c\in(1,(6a+5)/(4a+3)).$ Choosing $a=1/6,$ for $1<c<\frac{18}{11}=1.\dot{6}\dot{3}$ and   $S\leq N^{c-\varepsilon},$ we have
$
\mathfrak{Q}_{c}(S,N)=O(N^{1-\varepsilon}).
$
This gives  Corollary \ref{LSZ2} for the first case.
Specially, by involving some results in \cite{Bo}, for $1<c<\frac{249}{152}$ and   $S\leq N^{c-\varepsilon},$  one can get
\begin{align}\label{b}
\mathfrak{Q}_{c}(S,N)=O(N^{1-\varepsilon}).
\end{align}
In fact, assuming that $(a,1/2+a)$  is an exponential pair, then by using the relation
$BABA(a,1/2+a)=\left((2a+1)/(6a+5),(4a+3)/(6a+5)\right),$
for $c\in(1,(6a+5)/(4a+3))$ and   $S\leq N^{c-\varepsilon},$
we can get
$\mathfrak{Q}_{c}(S,N)\ll N^{1-\varepsilon}.$
By the work of Bourgain \cite{Bo}, we can choose $a=13/84+\varepsilon.$ This gives (\ref{b}).
\end{Remark}

{\bf Notation:}
The implied constants through out this paper may depend on $\varepsilon$ and the parameter $c$ (and
obviously, some times depend on some auxiliary parameters). However, these implied constants are always uniform with respect to the main parameters $N$, $s$ and $S$.
$\mathcal{Q}_{2}$ is
the set of positive square-free integers.
We write $u\sim U$ to denote $U<u\leq 2U.$
We use $\square$ to denote a non-specified integer square. This means that $n=\square$ is equivalent to the statement that $n$ is a perfect square and thus we can write
 \[Q_{c}(s,N)=\sum_{n\leq N\atop [n^{c}]=s\square}1.
 \]

\bigskip
\section{Proof for Theorem \ref{LSZ0}}

Before the  proof, we need introduce some lemmas as preliminaries.
The following lemma can be seen as a weaker version of Lemma 4.3 in \cite{GK} (see also Corollary 6.7 in \cite{Bor}).

\begin{Lemma}\label{I1}
Let $f\in C^{\infty}([N,2N])$ such that there exists $T>0$ such that, for all $x\in[N,2N]$ and all $j\in\mathbb{Z}_{\geq0},$ we have $|f^{(j)}(x)|\asymp T N^{-j}.$ Let $(\kappa,\lambda)$ be an exponent pair. Then
\[
\sum_{N<n\leq 2N}\psi(f(n))\ll \left(T^{\kappa}N^{\lambda}\right)^{1/(\kappa+1)}
+NT^{-1}.
\]
\end{Lemma}

By Lemma \ref{I1}, we can immediately obtain the follows.
\begin{Corollary}\label{cz}
Let $N\in\mathbb{Z}_{\geq1}$ be sufficiently large. Let $d\in\{1,2\},$ $1/2<\gamma<1,$ $|\sigma|\leq1,$  and $y\neq0$ fixed. If $(\kappa,\lambda)$ is an exponent pair, then one has
\[
\sum_{N<n\leq 2N}\psi(y(n^{d}+\sigma)^{\gamma})\ll |y|^{\frac{\kappa}{1+\kappa}}
N^{\frac{\lambda+d\kappa\gamma}
{1+\kappa}}+|y|^{-1}N^{1-d\gamma}.
\]
\end{Corollary}

Next lemma can be seen in Theorem A.6 in \cite{GK} or Theorem 18 in \cite{Va}.
\begin{Lemma}\label{z3}
For $0<|t|<1,$ let
$$W(t) = \pi t(1-|t|)\cot\pi t + |t|.$$
  For $x\in\mathbb{R},$ $H\geq1,$ we define
$$\psi^{*}(x)=\sum_{1\leq |h|\leq H}(2\pi ih)^{-1}W\left(\frac{h}{H+1}\right)e(hx)$$
and
\[
\delta(x)=\frac{1}{2H+2}\sum_{|h|\leq H}\left(1-\frac{|h|}{H+1}\right)e(hx).
\]
Then $\delta(x)$ is non-negative, and we have
$$|\psi^{*}(x)-\psi(x)|\leq
\delta(x).$$
\end{Lemma}

We follow  the arguments of    Theorem 2.3 in \cite{lsz}.
Let $c>1$, $c\notin \mathbb{N}$.
Then the condition $\left[n^c\right]=sm^r$ is equivalent to
$$
-(sm^r+1)^{\gamma}< -n\le -(sm^r)^{\gamma}.
$$
Recall that $\gamma=1/c$.
Let $M=s^{-1/r}N^{s/r}$. So with
$$
\mathfrak{Q}_{c}(S,N)=\sum_{\substack{s\leq S\\ s\in\mathcal{Q}_{2}}} Q_{c}(s,N)
,$$
 we have
\begin{align}\label{SZ1}
\mathfrak{Q}_{c}(S,N)=
S_{0}-E_{1}+E_{2}+O(1),
\end{align}
where
\begin{align*}
S_{0}&=\sum_{\substack{ sm^r\leq N^{c}\\s\leq S\\
s\in\mathcal{Q}_{r}} }
\left((sm^r+1)^{\gamma}
-(sm^r)^{\gamma}\right)
\end{align*}
and
\[
E_1=\sum_{\substack{sm^r\leq N^{c}\\s\leq S\\
s\in\mathcal{Q}_{r}}}
\psi\left(-(sm^r)^{\gamma}\right),\quad
E_2=\sum_{{\substack{sm^r\leq N^{c}\\s\leq S\\
s\in\mathcal{Q}_{r}}}}
\psi\left(-(sm^r+1)^{\gamma}\right).
\]
By following the arguments in \cite{lsz}, we have
\begin{align*}
S_{0}&=\sum_{\substack{sm^{2}\leq N^{c}, s\leq S\\ s\in\mathcal{Q}_{2}}}\left((sm^{2}+1)^{\gamma}
-s^{\gamma}m^{2\gamma}\right)\\
&=\frac{12\gamma}{\pi^{2}(2\gamma-1)}
S^{\frac12}N^{1-c/2}
+O(N^{1-c/2}\log N)
+O(SN^{1-c}).
\end{align*}
Then we focus on the estimates of
\[
E_{1}=\sum_{\substack{sm^{2}\leq N^{c}, s\leq S\\ s\in\mathcal{Q}_{2}}}\psi(-s^{\gamma}
m^{2\gamma})
\]
and
\[
E_{2}=\sum_{\substack{sm^{2}\leq N^{c}, s\leq S\\ s\in\mathcal{Q}_{2}}}
\psi\left(-(sm^{2}+1)^{\gamma}\right).
\]

Then we can complete the proof for the following three lemmas.
For the proof, we will deal with the exponential sum by three different ideas. Firstly, we   deal with $E_{1}$ and $E_{2}$ by using Corollary \ref{cz}. In fact, with this idea, we can obtain the following lemma.
\begin{Lemma}\label{001}
For any fixed $c\in(1,2),$ $\gamma=1/c,$   any exponent pair $(\kappa,\lambda)$ with $(1-\kappa)/2\leq\lambda-\kappa$   and  $N\rightarrow\infty$,  we have
\begin{align*}
\begin{split}
\mathfrak{Q}_{c}(S,N)-
\frac{12\gamma}{\pi^{2}(2\gamma-1)}
S^{1/2}N^{1-c/2}
 \ll
N^{\frac{\kappa}{1+\kappa}+\frac c2+\varepsilon}S^{\frac{\lambda}{1+\kappa}-\frac12}
 +N^{\varepsilon}S^{1-\gamma}+SN^{1-c+\varepsilon}.
\end{split}
\end{align*}
\end{Lemma}
\begin{proof}
For $E_{1}$, we have
\begin{align*}
E_{1}&=\sum_{\substack{sm^{2}\leq N^{c}, s\leq S\\s\in\mathcal{Q}_{2}}}\psi(-s^{\gamma}
m^{2\gamma})\\
&=\sum_{rd^{2}m^{2}\leq N^{c}\atop rd^{2}\leq S}\mu(d)\psi(-r^{\gamma}
d^{2\gamma}m^{2\gamma})\\
&=\sum_{m\leq N^{c/2}}\sum_{d\leq N^{c/2}/m\atop d\leq S^{1/2}}\mu(d)
\sum_{r\leq N^{c}/m^{2}d^{2}\atop r\leq S/d^{2}}
\psi(-r^{\gamma}
d^{2\gamma}m^{2\gamma})
.\end{align*}
By Corollary \ref{cz}, this gives that
\begin{align*}
E_{1}&\ll (\log N)^{3}\left(\sum_{m\sim M}\sum_{d\sim D}\left|\sum_{r\sim R}
\psi(-r^{\gamma}
d^{2\gamma}m^{2\gamma})\right|\right)\\
& \ll (\log N)^{3}\left(\sum_{m\sim M}\sum_{d\sim D}\left(\left(d^{2\gamma}
m^{2\gamma}\right)^{\frac{\kappa}{1+\kappa}}
R^{\frac{\lambda+\gamma\kappa}{1+\kappa}}+
\left(d^{2\gamma}
m^{2\gamma}\right)^{-1}R^{1-\gamma}
\right)\right),
\end{align*}
where
\[
RD^{2}M^{2}\ll N^{c} \ \textup{and} \ \ RD^{2}\ll S.
\]
Hence for $(1-\kappa)/2\leq\lambda-\kappa,$ we have
\[
E_{1}\ll N^{\frac{\kappa}{1+\kappa}+\frac c2+
\varepsilon}S^{\frac{\lambda}{1+\kappa}-\frac12}
+S^{1-\gamma}.
\]
Similar argument gives that for $(1-\kappa)/2\leq\lambda-\kappa,$
\[
E_{2}\ll N^{\frac{\kappa}{1+\kappa}+\frac c2+
\varepsilon}S^{\frac{\lambda}{1+\kappa}-\frac12}
+S^{1-\gamma}.
\]
Recall that
\begin{align*}
S_{0}&=\sum_{\substack{sm^{2}\leq N^{c}, s\leq S\\ s\in\mathcal{Q}_{2}}}\left((sm^{2}+1)^{\gamma}
-s^{\gamma}m^{2\gamma}\right)\\
&=\frac{12\gamma}{\pi^{2}(2\gamma-1)}
S^{\frac12}N^{1-c/2}
+O(N^{1-c/2}\log N)
+O(SN^{1-c}).
\end{align*}
Then for $(1-\kappa)/2\leq\lambda-\kappa,$ we can obtain that
\begin{align*}
\mathfrak{Q}_{c}(S,N)-
\frac{12\gamma}{\pi^{2}(2\gamma-1)}
&S^{1/2}N^{1-c/2}
\ll N^{\frac{\kappa}{1+\kappa}+\frac c2+\varepsilon}S^{\frac{\lambda}{1+\kappa}-\frac12}
+N^{\varepsilon}S^{1-\gamma}
\\&
+SN^{1-c}+N^{1-\frac c2+\varepsilon}.\end{align*}
One may note that $\frac{\lambda}{1+\kappa}-\frac12>0$ and $1-\frac{c}{2}<\frac{\kappa}{1+\kappa}+\frac{c}{2}$ for $c\in(1,2).$
Hence, for $c\in(1,2),$ we have
\begin{align*}
\mathfrak{Q}_{c}(S,N)-
\frac{12\gamma}{\pi^{2}(2\gamma-1)}
S^{1/2}N^{1-c/2}
\ll N^{\frac{\kappa}{1+\kappa}+\frac c2+\varepsilon}S^{\frac{\lambda}{1+\kappa}-\frac12}
+SN^{1-c}+N^{\varepsilon}S^{1-\gamma}.
\end{align*}
\end{proof}
The second idea is that we can change quadruple sum into  triple sum by interesting a divisor function. This is possible. Because  the indexes of two variables are the same. Hence, we can save a sum.
\begin{Lemma}\label{002}
For any fixed $c\in(1,2),$ $\gamma=1/c,$     and  $N\rightarrow\infty$,  we have
\begin{align*}
\begin{split}
\mathfrak{Q}_{c}(S,N)-
\frac{12\gamma}{\pi^{2}(2\gamma-1)}
S^{1/2}N^{1-c/2}
 \ll
S^{1/8}N^{(2+3c)/8+\varepsilon}
 +SN^{1-c+\varepsilon}.
\end{split}
\end{align*}
\end{Lemma}

\begin{proof}
Firstly, we introduce a well-known result in \cite{RS}.
\begin{Lemma}\label{rs}
For real numbers $\alpha_{1}, \alpha_{2}, \alpha_{3}$ such that $\alpha_{1}\alpha_{2}\alpha_{3}
(\alpha_{1}-1)(\alpha_{2}-2)\neq0$.
For $X>0,$ $M_{1}\geq1,$ $M_{2}\geq 1,$ and $M_{3}\geq 1,$ let
\[
S(M_{1},M_{2},M_{3}):=
\sum_{m_{2}\sim M_{2}}\sum_{m_{3}\sim M_{3}}
\left|\sum_{m_{1}\sim M_{1}}e\left(X\frac{m_{1}^{\alpha_{1}} m_{2}^{\alpha_{2}}
 m_{3}^{\alpha_{3}}}
{M_{1}^{\alpha_{1}}M_{2}^{\alpha_{2}}
M_{3}^{\alpha_{3}}}
\right)\right|,
\]
where $e(t)=e^{2\pi i t}.$   For any $\varepsilon>0,$ we have
\begin{align*}
S(M_{1},M_{2},M_{3})(XM_{1}M_{2}M_{3})^{-\varepsilon}&\ll
 \left(X
M_{1}^{2}M_{2}^{3}
M_{3}^{3}\right)^{1/4} +M_{1}^{1/2} M_{2}M_{3}
 +X^{-1}M_{1}M_{2}M_{3},
\end{align*}
 where
the implied constant may depend on $\alpha_{1},\alpha_{2},\alpha_{3},$ and $\varepsilon.$
\end{Lemma}
For $E_{1}$, we have
\begin{align*}
E_{1}&=\sum_{\substack{sm^{2}\leq N^{c}, s\leq S\\s\in\mathcal{Q}_{2}}}\psi(-s^{\gamma}
m^{2\gamma})\\
&=\sum_{rd^{2}m^{2}\leq N^{c}\atop rd^{2}\leq S}\mu(d)\psi(-r^{\gamma}
d^{2\gamma}m^{2\gamma})\\
&=\sum_{m\leq N^{c/2}}\sum_{d\leq N^{c/2}/m\atop d\leq S^{1/2}}\mu(d)
\sum_{r\leq N^{c}/m^{2}d^{2}\atop r\leq S/d^{2}}
\psi(-r^{\gamma}
d^{2\gamma}m^{2\gamma})
.\end{align*}
By Lemma \ref{z3}, this gives that
\begin{align*}
E_{1}&\ll N^{\varepsilon}\left(\sum_{m\sim M}\sum_{d\sim D}\left|\sum_{r\sim R}
\psi(-r^{\gamma}
d^{2\gamma}m^{2\gamma})\right|\right)\\
& \ll
 N^{\varepsilon}\left(\sum_{h\sim H}h^{-1}\sum_{t\sim T}\left|\sum_{r\sim R}
e\left(-hr^{\gamma}
t^{2\gamma}\right)\right|\right)+H^{-1}S^{1/2}N^{c/2},
\end{align*}
where
\[
T\asymp DM,\ RD^{2}M^{2}\ll N^{c} \ \textup{and} \ \ RD^{2}\ll S.
\]
Then by Lemma \ref{rs}, we have
\begin{align*}
E_{1}&\ll N^{\varepsilon}\left(\sum_{m\sim M}\sum_{d\sim D}\left|\sum_{r\sim R}
\psi(-r^{\gamma}
d^{2\gamma}m^{2\gamma})\right|\right)\\
& \ll
S^{1/8}N^{(2+3c)/8+\varepsilon} +SN^{1-c+\varepsilon}.
\end{align*}
Similar argument gives that
\[
E_{2}\ll S^{1/8}N^{(2+3c)/8+\varepsilon} +SN^{1-c+\varepsilon}.
\]
Recall that
\begin{align*}
S_{0}&=\sum_{\substack{sm^{2}\leq N^{c}, s\leq S\\ s\in\mathcal{Q}_{2}}}\left((sm^{2}+1)^{\gamma}
-s^{\gamma}m^{2\gamma}\right)\\
&=\frac{12\gamma}{\pi^{2}(2\gamma-1)}
S^{\frac12}N^{1-c/2}
+O(N^{1-c/2}\log N)
+O(SN^{1-c}).
\end{align*}
Note that for $c\in(1,2),$ one has $(1-c/2)<(2+3c)/8.$
Then   we can obtain that
\begin{align*}
\mathfrak{Q}_{c}(S,N)-
\frac{12\gamma}{\pi^{2}(2\gamma-1)}
S^{1/2}N^{1-c/2}
\ll S^{1/8}N^{(2+3c)/8+\varepsilon} +SN^{1-c+\varepsilon}
.\end{align*}

\end{proof}

The following lemma is  using only one of these ideas. Then by following the arguments in \cite{lsz}, introducing Lemma \ref{rs} and inserting some ideas of \cite{Bor,GK,zw}, we can obtain the following lemma.
\begin{Lemma}\label{003}
For any fixed $c\in(1,2),$ $\gamma=1/c,$     and  $N\rightarrow\infty$,  we have
\begin{align*}
\begin{split}
\mathfrak{Q}_{c}(S,N)-
\frac{12\gamma}{\pi^{2}(2\gamma-1)}
S^{1/2}N^{1-c/2}
 \ll
S^{1/5}N^{(1+2c)/5+\varepsilon}
 +SN^{1-c+\varepsilon}.
\end{split}
\end{align*}
\end{Lemma}
\begin{proof}
For $E_{1}$, we have
\begin{align*}
E_{1}&=\sum_{\substack{sm^{2}\leq N^{c}, s\leq S\\s\in\mathcal{Q}_{2}}}\psi(-s^{\gamma}
m^{2\gamma})\\
&=\sum_{rd^{2}m^{2}\leq N^{c}\atop rd^{2}\leq S}\mu(d)\psi(-r^{\gamma}
d^{2\gamma}m^{2\gamma})\\
&=\sum_{m\leq N^{c/2}}\sum_{d\leq N^{c/2}/m\atop d\leq S^{1/2}}\mu(d)
\sum_{r\leq N^{c}/m^{2}d^{2}\atop r\leq S/d^{2}}
\psi(-r^{\gamma}
d^{2\gamma}m^{2\gamma})
.\end{align*}
By Lemma \ref{z3}, we can obtian that
\begin{align*}
E_{1}&\ll (\log N)^{3}\left(\sum_{m\sim M}\sum_{d\sim D}\left|\sum_{r\sim R}
\psi(-r^{\gamma}
d^{2\gamma}m^{2\gamma})\right|\right)\\
& \ll
 N^{\varepsilon}\sum_{h\sim H}h^{-1}\left(\sum_{m\sim M}\sum_{d\sim D}\left|\sum_{r\sim R}
e(-hr^{\gamma}
d^{2\gamma}m^{2\gamma})\right|\right)+H^{-1}S^{1/2}N^{c/2},
\end{align*}
where
\[
RD^{2}M^{2}\ll N^{c} \ \textup{and} \ \ RD^{2}\ll S.
\]
By Lemma \ref{rs}, we have
\begin{align*}
 \sum_{m\sim M}\sum_{d\sim D}\left|\sum_{r\sim R}
e(-hr^{\gamma}
d^{2\gamma}m^{2\gamma})\right|
&\ll \left(\left(hR^{\gamma}D^{2\gamma}M^{2\gamma}\right)^{1/4}
R^{1/2}(DM)^{3/4}\right.
\\&\left.+R^{1/2}DM
+RDM/hR^{\gamma}D^{2\gamma}M^{2\gamma}\right)
(HDM)^{\varepsilon}\\
&\ll
\left(h^{1/4}S^{1/8}N^{1/4+3c/8}
+N^{c/2}+S^{1/2}N^{c/2-1}h^{-1}\right)N^{\varepsilon}.
\end{align*}
Summing over $h,$ we can obtain that
\begin{align*}
E_{1} &\ll\sum_{h\sim H}h^{-1}\left(h^{1/4}S^{1/8}N^{1/4+3c/8}
+N^{c/2}+S^{1/2}N^{c/2-1}h^{-1}\right)N^{\varepsilon}
\\&+H^{-1}S^{1/2}N^{c/2}.
\end{align*}
Choosing $H=[S^{3/10}N^{c/10-1/5}],$ then we have
\[
E_{1} \ll S^{1/5}N^{(1+2c)/5+\varepsilon}+SN^{1-c+\varepsilon}
\]
by assuming such that $[S^{3/10}N^{c/10-1/5}]\geq1.$
For $[S^{3/10}N^{c/10-1/5}]<1,$ we may observe that the estimate of $S^{1/5}N^{(1+2c)/5+\varepsilon}$ is trivial.
Similar argument gives that
\[
E_{2}\ll S^{1/8}N^{(2+3c)/8+\varepsilon} +SN^{1-c+\varepsilon}.
\]
Recall that
\begin{align*}
S_{0}&=\sum_{\substack{sm^{2}\leq N^{c}, s\leq S\\ s\in\mathcal{Q}_{2}}}\left((sm^{2}+1)^{\gamma}
-s^{\gamma}m^{2\gamma}\right)\\
&=\frac{12\gamma}{\pi^{2}(2\gamma-1)}
S^{\frac12}N^{1-c/2}
+O(N^{1-c/2}\log N)
+O(SN^{1-c}).
\end{align*}
Note that for $c\in(1,2),$ one has $(1-c/2)<(1+2c)/5.$
Then   we can obtain that
\begin{align*}
\mathfrak{Q}_{c}(S,N)-
\frac{12\gamma}{\pi^{2}(2\gamma-1)}
S^{1/2}N^{1-c/2}
\ll S^{1/5}N^{(1+2c)/5+\varepsilon} +SN^{1-c+\varepsilon}
.\end{align*}
\end{proof}

\bigskip
{\bf Acknowledgements} The authors would like to thank the referee who gives some  detailed corrections and suggestions.


\address{Wei Zhang\\ School of Mathematics and Statistics\\
               Henan University\\
               Kaifeng  475004, Henan\\
               China}
\email{zhangweimath@126.com}
\end{document}